\documentclass[reqno,12pt]{article}
\usepackage{mathtext}
\usepackage[utf8x]{inputenc}
\usepackage[english,russian]{babel}
\usepackage[T2A]{fontenc}
\usepackage{citehack} 
\usepackage{amsthm,amsmath}    
\usepackage{graphicx}
\usepackage{wrapfig}
\usepackage{amssymb}
\usepackage{amsfonts}

\usepackage{amsbib}

\title{$X$- и $Y$-инварианты дифференциальных операторов с частными производными на плоскости}
\author{Екатерина Шемякова}

\theoremstyle{plain}

\newtheorem{theorem}{Теорема}[section]
\newtheorem{lemma}[theorem]{Лемма}
\newtheorem{corollary}[theorem]{Следствие}
\newtheorem{proposition}[theorem]{Утверждение}

\theoremstyle{definition}
\newtheorem{definition}[theorem]{Определение}
\newtheorem{example}[theorem]{Пример}

\theoremstyle{remark}

\renewenvironment{proof}{\noindent{\it Доказательство}.}{\qed}

\newcommand{\Sym}{\ensuremath \mathrm{Sym}}

\newcommand{\Ker}{\ensuremath \mathop \mathrm{Ker} \nolimits}

\date{}

\begin{document}
\maketitle

\begin{abstract} В работе рассматривается классическая проблема компьютерной алгебры -- символьное решение
дифференциальных уравнений. А именно, широко используемые теоремы Дарбу о 
преобразованиях гиперболических операторов на плоскости с помощью дифференциальных подстановок переносятся в пространство инвариантов.
Вводятся $X$- и $Y$-инварианты для таких операторов как решения некоторых уравнений записанных в терминах инвариантов Лапласа
самого оператора $L$ относительно калибровочных преобразований. Получены явные формулы преобразований множеств $X$- и $Y$-инвариантов 
при преобразованиях Дарбу. 
\end{abstract}

\section{Введение}

Теория преобразований Дарбу для линейных гиперболических уравнений второго порядка на плоскости 
является классикой символьных алгоритмов и имеет многочисленные приложения: для символьных решений  
задач классической дифференциальной геометрии \cite{bianchi,Eis}, теории
интегрируемых систем \cite{fer-lap,NV}, 
для исследования нелинейных уравнений интегрируемых по Дарбу~\cite{2ZhS2001_rus,anderson_kamran97,forsyth,Goursat}.
Получены многочисленные обобщения классической теории Дарбу как для гиперболических задач
(для систем уравнений на плоскости~\cite{Zh-St,St08,Tsarev99,ts:genLaplace05}, 
для уравнений с более чем двумя независимыми переменными~\cite{Dini1901,Dini1902,ts:2007:dini}, 
так и для негиперболических:~\cite{Roux1899,petren,pisati,2ndorderparab_rus}.

В данной работе мы возвращаемся к классическим для теории Дарбу
гиперболическим операторам второго порядка, т.е. операторам вида
\begin{equation} \label{op:L}
L  = D_{x} D_y + a(x,y) D_x + b(x,y) D_y + c(x,y)\ . \
\end{equation}
Пусть $u=u(x,y)$, $L(u)=0$, $M$ -- линейный дифференциальный оператор с частными производными,
и $v(x,y)=Mu$. В общем случае такое $v$ будет удовлетворять переопределенной 
системе линейных дифференциальных уравнений, и лишь
при особом выборе $M$ мы получаем только одно новое уравнение: $L_1 v =0$, где $L_1$ того же вида что и $L$, 
(\ref{op:L}), но с измененными коэффициентами $a_1(x,y)$, $c_1(x,y)$, $b_1\equiv b$. 
В таком случае мы имеем \emph{дифференциальное преобразование}~\cite{2ndorderparab_rus} 
оператора $L$ в $L_1$ с помощью $M$, и для некоторого оператора $M_1$
имеет место равенство
\begin{equation} \label{main}
M_1 \circ L = L_1 \circ M \ ,
\end{equation}
задающее левое наименьшее общее кратное $lLCM(L,M)$ в кольце $K[D]=K[D_x,
D_y]$ линейных дифференциальных операторов на плоскости.

Известно~\cite[Ch.
VIII]{Darboux2}, что в случае общего положения все такие операторы $M$
описываются в терминах одного или двух частных решений уравнения $L(z)=0$. Возможны и вырожденные 
случаи, среди которых и классическое 
преобразование Лапласа, для задания которого требуется лишь знание коэффициентов
оператора (\ref{op:L}). Соотношение (\ref{main}) для
``сплетающего оператора'' $M$ также широко использовалось при
изучении интегрируемых задач~\cite{VSh93,BV00}. 

В данной работе мы переносим классические результаты Дарбу о 
преобразованиях гиперболических операторов на плоскости с помощью дифференциальных подстановок в пространство инвариантов.
Согласно Дарбу любое дифференциальное преобразование может быть представлено в виде последовательности 
дифференциальных преобразований простейшего вида, т.е. таких что $M$ -- оператор первого порядка и 
содержит операторы дифференцирования только по одной переменной (мы называем их $X$- и $Y$-преобразования Дарбу). 
Так как частные решения $z$ уравнения $L(z)=0$ используются для построения дифференциальных
преобразований, мы рассматриваем инварианты $R$ и $Q$,~(\ref{def:RQ}) пар $(L,z)$ вместо обыкновенно рассматриваемых порождающих инвариантов $h$ и $k$
самого оператора $L$ относительно калибровочных (см. определение в секции~\ref{sec:def}) преобразований. Функции $Q$ и $R$ можно построить по $z$,
и обратно, по $R$ и $Q$ можно построить $z$.

Мы показываем, что каждую из функций $R$ и $Q$ можно задать как решение некоторого уравнения ((\ref{eq:r}) и (\ref{eq:q})) с коэффициентами
зависящими только от $h$ и $k$, то есть уравнение одно для каждого класса эквивалентности операторов $L$. 
Решения этих уравнений мы называем $X$- и $Y$-инвариантами. Свойства $X$- и $Y$-инвариантов показывают, что введенные понятия 
интересны и могут быть полезными для развития методов Дарбу и их обобщений.

Автор благодарит профессора С.П.Царева за полезные замечания. 

\section{Основные определения}
\label{sec:def}

Пусть $K$ -- дифференциальном поле $K$ характеристики ноль с коммутирующими дифференцированиями $\partial_x, \partial_y$,
$K[D]=K[D_x, D_y]$ -- кольцо $K[D]=K[D_x, D_y]$ линейных операторов с частными
производными с коэффициентами из $K$, где $D_x, D_y$ соответствуют дифференцированиям $\partial_x, \partial_y$.  

Операторы $L \in K[D]$ имеют вид $L = \sum_{i+j =0}^d a_{ij}
D_x^i D_y^j$, где $a_{ij} \in K$. Многочлен $\Sym_L =  \sum_{i+j =
d} a_{ij} X^i Y^j$ формальных переменных $X, Y$ назовем (главным)
\emph{символом} $L$.

Поле  $K$ будем предполагать дифференциально замкнутым, т.е.\ содержащим решения (нелинейных в общем случае)
дифференциальных уравнений с коэффициентами из $K$.

Пусть $K^*$ -- множество обратимых элементов из $K$. 
Для $L \in K[D]$ и каждого $g \in K^*$ рассмотрим калибровочное~\footnote{в англ.gauge transformation} преобразование
\begin{equation} \label{def:gauge}
 L \rightarrow L^g = g^{-1} \circ L \circ g \ .
\end{equation}
Тогда дифференциально-алгебраическое выражение $I$ от коэффициентов
оператора $L$ и их производных называется \emph{инвариантом}
относительно этих калибровочных преобразований, если оно не меняется
при этих преобразованиях. Простейшие примеры инвариантов --
коэффициенты символа оператора. Порождающим множеством инвариантов
называется набор инвариантов, такой, что любой дифференциальный
инвариант может быть выражен через инварианты этого набора и их
производные.

\section{X- и Y-инварианты}

Известно~\cite{Darboux2}, что функции 
\begin{equation} \label{def:h,k}
h=ab+a_x - c \ , \  k=ab +b_y - c
\end{equation}
образуют порождающее множество дифференциальных инвариантов 
для операторов $L$ вида~(\ref{op:L}) относительно калибровочных преобразований~(\ref{def:gauge}). Соответствующее 
преобразование ядер, $\Ker(L) \rightarrow \Ker(L')$: 
\begin{equation} \label{z->z'}
z \mapsto z'= \frac{z}{g} \ . 
\end{equation}
Функции $h$ и $k$ называют инвариантами Лапласа.

Далее мы будем рассматривать калибровочные преобразования пар $(z, L)$, где $z \in \Ker L$.

\begin{lemma} Функции
\begin{equation} \label{def:RQ}
R(b,z)=R=-b-\frac{z_x}{z} \quad  \text{и} \quad  
Q(a,z)=Q=-a- \frac{z_y}{z} 
\end{equation} 
 являются инвариантами для пар $(z, L)$, $z \neq 0$ относительно калибровочных преобразований. 
\end{lemma}
\begin{proof} Достаточно выписать формулы преобразований коэффициентов 
операторов $L$ вида~(\ref{op:L}) и использовать~(\ref{z->z'}).
\end{proof}

Функции $R$ и $Q$ обладают примечательным свойством: они удовлетворяют уравнениям, записанным только в терминах $h$ и $k$. Таким образом,
уравнения не содержат $z$.

\begin{proposition}  Пусть $R=r \neq 0, Q=q \neq 0$ получены по формулам~(\ref{def:RQ})
из некоторого $z \in \Ker L, z \neq 0$, тогда верны соотношения
\begin{eqnarray}
h-k - r_y + \left(\frac{k}{r} \right)_x + (\ln r)_{xy}  &=& 0  \ , \label{eq:r} \\
h-k + q_x - \left(\frac{h}{q} \right)_y - (\ln q)_{xy}  &=& 0 \ , \label{eq:q} 
\end{eqnarray}
где $h$ и $k$ -- инварианта Лапласа~(\ref{def:h,k}) оператора $L$.
\end{proposition}
\begin{proof}
Равенства проверяются подстановкой.
\end{proof}

\begin{definition} Решения $r$ и $q$ уравнений ~(\ref{eq:r}) и ~(\ref{eq:q}) мы назовем 
$X$- и $Y$-инвариантами соответственно. 
\end{definition}

\begin{lemma}
\label{r_z_1_to_1}
Каждому $X$- (соотв. $Y$-) инварианту $r$ (соотв. $q$) соответствует
только одно (с точностью до умножения на константу) $z$
такое, что $z \in \Ker L$ и $r=-b-z_x/z$ (соотв. $q=-a-z_y/z$). То есть
\begin{eqnarray}
&& z = f(y) e^{- \int b+r \ dx} \label{z:f(y):exp} \\
&& \left( \text{соотв.}  \quad z = g(x) e^{- \int  a+q \ dy} \right) \ , \nonumber  
\end{eqnarray}
где $f(y)$ и $g(x)$ единственны с точностью до константы~\footnote{т.е. $f(y)$ и $g(x)$ здесь не параметры, а определенные функции}. 
\end{lemma}
\begin{proof}
Пусть $r$ -- $X$-инвариант некоторого оператора $L$, (\ref{op:L}).
Пусть $z = f(y) e^{- \int_{0}^x b+r \ dx}$. Здесь мы для краткости не вводим новую переменную
для интегрирования. Таким образом, $f(y)$ -- выделенная ``неопределенность'' интегрирования по $x$.
Докажем, что можно подобрать $f(y)$ так, что $z \in \Ker L$. 

Выразим коэффициенты $c=c(x,y)$ и $b=b(x,y)$ через инварианты Лапласа $h$
и $k$,~(\ref{def:h,k}) оператора $L$: $c=-h+ab+a_x$ и 
$b= \int k - h + a_x \ dy$, и $a=a(x,y)$. 
В выражении для $b$ содержится неопределенность (произвольная функция от $x$), но это не влияет
дальнейшие вычисления, так в них участвует только $b_y$.

Выразим $h$ из равенства~(\ref{eq:r}).
Подставим эти выражения в выражение для $z$. Тогда равенство $L(z)=0$ имеет вид
\[
 f(y) A + (f(y))_y -f(y) \int_0^x A_x \ dx = 0 \ , 
\]
где
\[
 A = A(x,y) = \frac{k}{r}+\frac{r_y}{r} + a \ .
\]
Так как $\int_0^x A_x \ dx= A(x,y) - A(0,y)$, то имеем
\[
(f(y))_y + f(y) A(0,y) = 0 \ .
\]
Отсюда $f(y)$ находится однозначно с точностью до умножения на константу,
и, как видно, действительно, является функцией зависящей только от переменной
$y$.

Из условия $r=-b-z_x/z$
видим, что любое такое $z$, если существует, имеет вид $z = f(y) e^{- \int_{0}^x b+r \ dx}$, что и
доказывает единственность.

Утверждение для $q$ ($Y$-инварианта оператора $L$) доказывается аналогично.
\end{proof}

Таким образом, доказано важное свойство $X$- и $Y$- инвариантов: 
 множество всех $R$, полученных из $z \in \Ker L$, $z \neq 0$, совпадает с множеством решений $r$ уравнения~(\ref{eq:r}), т.е.
с множеством $X$-инвариантов (аналогично для $Q$, $q$ и $Y$-инвариантов). То есть $X$- и $Y$- инварианты можно рассматривать как ``проективизацию''
пространства решений $L(z)=0$.  

\begin{corollary}[связь между $X$- и $Y$- инвариантами]
Между $X$- и $Y$-инвариантами существует взаимно-однозначное соответствие.

Такие пары $X$- и $Y$- инвариантов будем называть ``соответствующими''. Они удовлетворяют соотношению
\begin{equation} \label{equality:rqhk}
r_y - q_x  = h - k \ ,
\end{equation}
где $h$ и $k$ -- инварианты Лапласа~(\ref{def:h,k}) оператора $L$.
\end{corollary}
\begin{proof} Пусть $r$ -- $X$-инвариант, тогда по Лемме~\ref{r_z_1_to_1} существует 
только одно (с точностью до умножения на константу) $z \in \Ker L$ такое, что $r=-b-z_x/z$.
Используя это $z$ строим $Y$-инвариант $q=-a-z_y/z$. Отметим, что при умножении $z$ на константу
$q$ не изменится, таким образом он единственный для данного $r$. 

 Перекрестно дифференцируя имеющиеся равенства 
$r=-b-z_x/z$ и $q=-a-z_y/z$, получаем~(\ref{equality:rqhk}).

Точно такое же соотношение получается, если по данному $Y$-инварианту $q$
построить, используя Лемму~\ref{r_z_1_to_1}, $z \in \Ker L$ такое, что $q=-a-z_y/z$,
а по $z$ построить $X$-инвариант $r=-b-z_x/z$. 
\end{proof}

\section{Свойства X- и Y- инвариантов}
\label{sec:last}
Пусть 
\[
\Ker_X(L) = \Ker_X(h,k) \quad (\text{соотв.} \quad \Ker_Y(L) = \Ker_Y(h,k))  
\]
-- множество $X$- (соотв. $Y$-) инвариантов оператора $L$ с порождающими инвариантами~(\ref{def:h,k}).

Пусть $z \in \Ker L, z \neq 0$, $r$ и $q$ -- соответствующие $X$- и $Y$-инварианты, тогда, согласно  
Дарбу~\cite{Darboux2}, для оператора 
\begin{equation} \label{M:X-Darboux}
M=D_x-\frac{z_x}{z}=D_x+r+b 
\end{equation}
существуют единственно определенные операторы $M_1$ и $L_1$ такие, что выполняется~(\ref{main}).
Аналогично, для оператора 
\[
M=D_y-\frac{z_y}{z} = D_y + q + a 
\]
существуют единственно определенные операторы $M_1$ и $L_1$ такие, что выполняется~(\ref{main}).
Такие дифференциальные преобразования назовем $X$- и $Y$-преобразования Дарбу соответственно. 
Такие преобразования дают ${\mathbb R}$-линейное отображение ядра $L$ в ядро $L_1$. 
Оказалось, что это отображение можно перенести и на $X$- и $Y$- инварианты, хотя этот факт заранее не очевиден. 

\begin{theorem}[$\Ker_Z(L)$ под действием $Z$-преобразований Дарбу, $Z \in \{ X,Y \}$]
\label{thm:ker_z}
 Пусть $r_0 \in \Ker_X(L)$ (соотв. $q_0 \in \Ker_Y(L)$), и $L_{1}$ -- оператор 
полученный из оператора $L$ соответствующим $X$- (соотв. $Y$-) преобразованием Дарбу, тогда
$\Ker_X(L) / \{ r_0 \} \rightarrow \Ker_X(L_1)$ 
(соотв. $\Ker_Y(L) / \{ q_0 \} \rightarrow \Ker_Y(L_1)$) по следующей формуле:

\begin{eqnarray} \label{eq:r->r_X_X}
&& r \mapsto r + \left(\frac{r}{r_0}\right)_x \frac{r_0}{r_0-r} \\
\nonumber
&& \left(\text{соотв.} \quad q \mapsto q +  \left(\frac{q}{q_0}\right)_y \frac{q_0}{q_0-q} \right) \ .
\end{eqnarray}
\end{theorem} 
\begin{proof} Обозначим оператор $L$ через его коэффициенты: (\ref{op:L}). Пусть $r_0 \in \Ker_X(L)$, тогда 
\[
z_0=  f(y) e^{- \int b+r_0 \ dx} \in \Ker L \ ,
\]
и мы можем выразить $c$ через $r_0, a, f(y), b$.
Подставим $M=D_x+r_0+b$ и выражение для $c$ в~(\ref{main}), где 
$L_1$ и $M_1$ тоже введем через коэффициенты:
 $L_1  = D_{x} D_y + a_1(x,y) D_x + b_1(x,y) D_y + c_1(x,y)$, $M_1=D_x+ m_{100}(x,y)$.
Далее в доказательстве мы для краткости будем опускать обозначение зависимости коэффициентов от независимых переменных,
например, будем писать $a$ вместо $a(x,y)$.

Из полученного операторного равенства, сравнивая соответственные коэффициенты, последовательно выразим $a_1, b_1, c_1, m_{100}$.
В частности, мы получим 
\[
m_{100}=r_0+b-\frac{r_{0x}}{r_0} \ .
\]

Пусть $r \in \Ker_X(L)  / \{ r_0 \}$, тогда 
\[
 z= g(y) e^{- \int b+r \ dx} \in \Ker L \ ,
\]
 и $M(z) \in \Ker L_1$. Заметим, что для $r=r_0$ мы получим
$M(z)=0$. Тогда $X$-инвариант оператора $L_1$
можно построить по формуле:
\begin{equation} \label{expr:Xinv_L1}
\tilde{r}=-b_1 - \frac{M(z)_x}{M(z)} \ . 
\end{equation}
Упрощая это выражение, получим~(\ref{eq:r->r_X_X}).  

Утверждение для $Y$-инвариантов доказывается аналогично. 
\end{proof}

\begin{example} Рассмотрим оператор 
\[
L= D_x D_y + 1-x^2-x y \ ,
\]
и его $X$-инвариант
\begin{equation} \label{r0:expr}
 r_0=x+y \ . 
\end{equation}
Инварианты Лапласа оператора $L$ равны: $h=k=-1+x^2+yx$.

Пусть $L_1$ -- результат $X$-преобразования Дарбу соответствующего $r_0$.
Согласно формуле~(\ref{M:X-Darboux}), имеем
$M=D_x+r_0=D_x+x+y$. Результат такого Дарбу преобразования, оператор $L_1$ 
найдем из операторного равенства~(\ref{main}):
\[
L_1 = D_x D_y - \frac{1}{x+y} D_y -x^2-xy \ .
\]
Его инварианты Лапласа~(\ref{def:h,k}) имеют вид
\[
 h_1 = x^2 + xy \ , \quad k_1 = \frac{x^4+3x^3y+3x^2y^2+y^3x+1}{(x+y)^2} \ .
\]

Пусть $r$ -- какой-то другой $X$-инвариант оператора $L$. Согласно Теореме~\ref{thm:ker_z}, 
$r_1$, полученное по формуле
\begin{equation} \label{r1:example}
 r_1 = \frac{-(x+y)r^2+(x^2+2xy+y^2-1)r+r_x x+r_x y}{(x+y)(x+y-r)} \ .
\end{equation}
является $X$-инвариантом оператора $L_1$. 

Это можно и явно проверить. Для этого проверим выполнение равенства~(\ref{eq:r})
для оператора $L_1$ (то есть при подстановке $r_1$, $h_1$, $k_1$ вместо $r$, $h$, $k$ соответственно). 
Левая часть равенства~(\ref{eq:r}) 
есть некоторое алгебраическо-дифференциальное выражение $B$ зависящее от 
$r, x, y, r_x, r_y, r_{xx}, r_{xy}, r_{xxy}$.

Для его упрощения воспользуемся тем, что $r$ - $X$-инвариант оператора 
$L$, а значит имеет место равенство ~(\ref{eq:r}). Получим
\begin{equation} \label{1}
\frac{r^2 r_y - 2 r x-r y - r_x + r_x x^2 + r_x x y - r_{xy} r + r_y r_x}{r^2} = 0 \ .
\end{equation}

Далее используем стандартную для дифференциальной алгебры процедуру: 
из равенства~(\ref{1}) выразим $r_{xy}$ через $r, x, y, r_x, r_y$, а затем, дифференцируя, 
это равенство по $x$ и подставляя выражение для $r_{xy}$, получаем выражение для
$r_{xxy}$ через  $r, x, y, r_x, r_y, r_{xx}$. Теперь подставим найденные выражения для 
$r_{xy}$ и $r_{xxy}$ в $B$. После упрощений получаем $B=0$.
\end{example}

\begin{theorem}[$\Ker_X(L)$ под действием $Y$-преобразований Дарбу] 
Пусть $q_0 \in \Ker_Y(L)$, 
и $L_{1}$ -- оператор полученный из оператора $L$ соответствующим $Y$-преобразованием Дарбу.
Тогда отображение $\Ker_X(L) / \{ r_0 \} \rightarrow \Ker_X(L_1)$ задается формулой
\[
r \mapsto - \frac{q_{0x}+h-q_0 r}{q_0-q} \ ,
\]
где $q$ -- соответствующий $r$ $Y$-инвариант оператора $L$,
и $r_0$ -- соответствующий $q_0$ $X$-инвариант оператора $L$.
\end{theorem}
\begin{proof} Обозначим оператор $L$ через его коэффициенты: (\ref{op:L}). 
Коэффициент $c$ можно выразить через инвариант $h$, (\ref{def:h,k}).
Подставим $M=D_y+q_0+a$ и выражение для $c$ в~(\ref{main}), где 
$L_1$ и $M_1$ тоже введем через коэффициенты: 
$L_1  = D_{x} D_y + a_1(x,y) D_x + b_1(x,y) D_y + c_1(x,y)$, $M_1=D_x+ m_{100}(x,y)$. 
Далее в доказательстве мы для краткости будем опускать обозначение зависимости коэффициентов от независимых переменных.

Из полученного операторного равенства, 
сравнивая соответственные коэффициенты, последовательно выразим $a_1, b_1, c_1, m_{100}$.
В частности, мы получим 
\[
m_{100}=q_0+a-\frac{q_{0y}}{q_0} \ .
\] 

Пусть $r \in \Ker_X(L)$,  $q$ -- соответствующий $r$ $Y$-инвариант оператора $L$, $z$ -- соответствующий элемент ядра $L$. Тогда $M(z) \in \Ker L_1$
и  $X$-инвариант оператора $L_1$ можно построить по формуле~(\ref{expr:Xinv_L1}).
Полученное выражение, 
\begin{equation} \label{expr:last:thm}
\tilde{r}=-\frac{(ab+a_x+b q_0 + q_{0x} ) z+z_x q_0 + b z_y + z_x a + z_{xy}}{(a+q_0)z+z_y} 
\end{equation}
не упрощается так же просто как выражение для $X$-инварианта в доказательстве Теоремы~\ref{thm:ker_z}. 

Воспользуемся тем, что $z \in \Ker L$, выразим $z_{xy}$ через производные $z$ меньшего порядка и подставим в~(\ref{expr:last:thm}).
Выразим производные $z$ первого порядка через $z$ используя связь между $r$ и $q$:
\begin{equation} \label{expr:zx_zy}
z_x = -(b + r) z \ , \quad z_y = -(a + q) z \ ,
\end{equation}
и подставим в~(\ref{expr:last:thm}), и получим утверждение теоремы.
\end{proof}

\section*{Заключение}

В данной статье мы полностью перенесли действие преобразований Дарбу в пространство инвариантов.
Введены новые понятия $X$- и $Y$-инвариантов, которые обладают достаточно интересными свойствами,
чтобы предполагать, что их дальнейшее изучение позволит прояснить многие проблемы символьных
решений дифференциальных уравнений.

Действие классических преобразований Лапласа в пространстве инвариантов 
остается пока невыясненным.

\end{document}